\documentclass[a4paper,11pt,twoside]{amsart}
\usepackage{amssymb}
\usepackage{amsthm}
\usepackage{amsmath}
\usepackage[pdftex,pdfpagelabels]{hyperref}
\hypersetup{}
\usepackage[nameinlink,noabbrev]{cleveref}
\usepackage{mathtools}
\usepackage{indentfirst}
\usepackage{verbatim}
\usepackage{environ}
\usepackage{float}
\usepackage[stretch=10,shrink=10]{microtype}
\usepackage{enumitem}
\usepackage{relsize}
\usepackage{graphicx}
\usepackage{tikz}
\usetikzlibrary{calc}
\usepackage{comment}
\usepackage{colonequals}
\usepackage[margin=1in]{geometry}
\usepackage{mathrsfs}
\usepackage{bbm}
\usepackage{float}
\usepackage{xcolor}
\usepackage{pgfplots}
\usepackage{caption}
\usepackage{doi}
\usepackage{url}
\pgfplotsset{compat=1.18}
\usepackage{hyperref}
\usepackage{booktabs}
\usepackage{multirow}
\usepackage{tabularx}
\usepackage{array}

\hypersetup{
  colorlinks=true,
  citecolor=black,
  linkcolor=black,
  urlcolor=black
}
\usepackage{cleveref}
\usepackage[numbers]{natbib}
\usepackage{fancyhdr}
\makeindex
\newtheorem{theorem}{Theorem}[section]
\newtheorem{lemma}[theorem]{Lemma}

\newtheorem*{example*}{Example}

\numberwithin{equation}{section}

\newcommand{\bb}{\mathbb}

\usepackage{lipsum}
\let\oldcenter\center   
\let\oldendcenter\endcenter
\renewenvironment{center}{\setlength\topsep{0pt}\oldcenter}{\oldendcenter}
\title{Efficient lifts into odd dimensions of Spherical $5$- and $7$-designs with Balanced Halves}

\author{Konstantin Delchev}
\address{Institute of Mathematics and Informatics,
Bulgarian Academy of Sciences, Sofia, Bulgaria}
\email{kdelchev@math.bas.bg} 

\author{Boryan Yordanov}
\address{American College of Sofia, Sofia, Bulgaria}
\email{boryan.g.yordanov@gmail.com}

\begin{document}

\begin{abstract}
We propose constructions of spherical $5$- and $7$-designs in odd dimensions by lifting balanced halves of even-dimensional designs to several layers of suitable latitudes. For every odd dimension, we obtain an explicit spherical $5$-design with at most $60(d^2-1)$ points, attaining the optimal quadratic order in the dimension. Assuming Zauner’s conjecture, this bound improves to $7.5(d-1)^2$. For every odd $d \geq 9$, we also construct a spherical $7$-design with $110M$ points from a complex projective $3$-design with $M$ points in $\mathbb{CP}^{(d-3)/2}$. 
\end{abstract}
\maketitle

\section{Introduction}

Spherical $\tau$-designs, finite sets on the unit sphere over which every harmonic polynomial of degree at most $\tau$ sums to zero, were introduced by Delsarte, Goethals and Seidel, who also established the fundamental lower bound on their cardinality
\cite{DelsarteGoethalsSeidel1977}.  

For a spherical design of strength $\tau$ and dimension $d$, this bound is
\[
|C|\geq \begin{cases}
\displaystyle \binom{d+e-1}{d-1}+\binom{d+e-2}{d-1}, & \tau=2e,\\[2mm]
\displaystyle 2\binom{d+e-1}{d-1}, & \tau=2e+1.
\end{cases}
\] 

This bound has order of $d^{\lfloor \tau/2\rfloor}$ and can be sharp, however tight designs, those attaining the bound, are extremely rare.  Bannai and Damerell proved severe restrictions on their possible strengths and dimensions \cite{BannaiDamerell1979,BannaiDamerell1980} with only certain small dimensions and the case $d=(2m+1)^2-2$ still open for $5$-designs. Further improvements of the lower bounds were made in a number of cases, notably through linear-programming methods of Boyvalenkov and collaborators \cite{BoyvalenkovNikova1998,BoyvalenkovDanevNikova1999, BoyvalenkovStoyanova2010}, but the gap between the available lower bounds and explicit constructions has remained large and, until recently, persisted even at the level of asymptotic order.

For spherical $5$-designs the Delsarte--Goethals--Seidel bound is of quadratic order in the dimension and  a number of earlier constructions achieved strong bounds in special families or under additional algebraic hypotheses; see, for example, the constructions of Kuperberg \cite{Kuperberg2006Codes} and of Mohammadpour and Waldron \cite{MohammadpourWaldron2024}.  More recently, Misawa gave explicit constructions of spherical $5$- and $7$-designs by means of tight fusion-frame and product constructions \cite{Misawa2026}.  Finally Arman et al. gave equal-weight spherical $5$-designs with $O(d^2)$ points in every dimension \cite{ArmanEtAl2026}.  In contrast, for spherical $7$-designs the Delsarte--Goethals--Seidel bound is cubic, and a general equal-weight construction with $O(d^3)$ points is still not proven to exist.

Layered constructions and different lifts have long played an important role in the theory of spherical designs. Bajnok developed systematic constructions obtained by decomposing spherical harmonics and placing lower-dimensional designs on suitable latitudes \cite{Bajnok1989,Bajnok1992TDesigns}, while Reznick gave a number of explicit constructions of spherical $5$-designs based on related decompositions and moment identities \cite{Reznick1995}.  Kuperberg's cubature constructions provide another general lifting perspective \cite{Kuperberg2006HatBox}.  There is also a complementary operation in the opposite direction: suitably chosen sections of a spherical design may be projected to produce lower-dimensional derived codes, already present in the work of Delsarte, Goethals and Seidel \cite{DelsarteGoethalsSeidel1977}. 

Starting points for the construction of spherical designs are diverse - graphs with exceptional spectral properties, codes and designs in discrete spaces, Grassmannian and complex projective constructions, and, classically, lattices. Modular-form methods show that shells of several highly symmetric lattices, and in some cases suitable unions of shells of a lattice and its rescaled dual, carry spherical designs of unexpectedly high strength \cite{HarpePacheVenkov2007}. These include a number of exceptional low-dimensional constructions which will serve as inputs to our lifting procedure.

Complex projective designs are an even more useful source. Via the Hopf fibration, a complex projective $t$-design in $\mathbb{CP}^{n-1}$, together with a suitable design (usually two-dimensional) on the fibres, can be lifted to a spherical design in $S^{2n-1}$, as demonstrated by Lindblad \cite{Lindblad2026}. In particular, projective $2$- and $3$-designs naturally give rise to spherical (5)- and (7)-designs, respectively. Of special interest are SIC-POVMs, the complex projective equivalent of maximal equiangular lines, which are thus $n^2$-point complex projective $2$-designs \cite{KlappeneckerRoetteler2005}. Most excitingly, Radchenko shared a seeming proof of the Zauner's conjecture (coming from a solution of Hilbert's 12th problem for real quadratic fields), which was obtained from an earlier raw AI-assisted work\cite{Radchenko2026Stark}.  While we believe the proof is correct, in the current paper we also mention the existing results independent of this result. 

A direct consequence from the hypothesis is the existence of a spherical $5$-design through Lindblad's construction with $(3/2)d^2$ points for every even dimension. In comparison, earlier in 2026 Arman, Radchenko et al. \cite{ArmanEtAl2026} provided a design with  $12d(d+2)$ points.

In this paper we provide a lift to an odd dimension for both cases, that will give a $5$-design with $(15/2)(d-1)^2$ and $60(d^2-1)$ points respectively. Earlier lift from \cite{ArmanEtAl2026} gives $72(d^2-1)$.

Our starting point is inspired by Bannai et al. \cite{ZhuBannaiEtAl2017, BannaiZhaoEtAl2018} - designs of harmonic indices $\{1,2,4\}$ obtained as balanced halves from a $5$-design. Both SIC-POVMs and the results of Arman et al. can supply such a design in an even-dimensional space, a construction which places a regular hexagon on each complex projective fiber \cite{ArmanEtAl2026, Lindblad2026}. Taking alternating vertices on each hexagon decomposes the design into two balanced halves, which we then lift as five antipodal copies on five pairs of latitudes. Although there is no general construction method for $O(d^3)$ even-dimensional $7$-designs, we describe a similar lift for $7$-designs in odd dimensions.

The paper is organized as follows - we present the two constructions first, proving their design properties conditional on the existence of a solution of two systems of equations. We then give the solutions of those two systems, which are mostly technical in nature. Finally we give tables with the best constructions for lower dimensions to our knowledge.

\section{Main Results}

\subsection{Construction of $5$-designs}

\begin{theorem}
    For every odd integer $d\geq25$ there exists an explicit equal-weight spherical $5$-design $C_d\subset S^{d-1}\subset\bb{R}^{d}$ with 
    \[
    |C_d|\leq60(d^2-1).
    \]
    If Zauner's conjecture for the existence of SIC-POVMs holds, there exists an equal-weight spherical $5$-design $C_d\subset S^{d-1}\subset\bb{R}^{d}$ with 
    \[
    |C_d|\leq\frac{15}{2}(d-1)^2.
    \]
\end{theorem}
\begin{proof}
Let us first describe the candidate design. We start from one half of the even-dimensional design in \cite{ArmanEtAl2026}, where $Y\subset S^{d-2}$ is chosen so that it is a balanced set and we form several antipodal latitude copies of $Y$ in $S^{d-1}$.  More precisely, for each of the five latitude parameters $u_i\in(0,1)$ and signs $\delta_i\in\{\pm1\}$, we include the points
\[
\pm(\sqrt{1-u_i^2}y,\delta_iu_i), \qquad y\in Y.
\]
Thus $u_i$ determines the absolute height of the $i$-th layer in the new coordinate, while $\delta_i$ determines which half, $Y$ or $-Y$, is placed at the positive latitude.  The resulting set is automatically antipodal, and we will show that the parameters $u_i$ and $\delta_i$ can be chosen so that the required second- and fourth-degree moments in the new coordinate are correct and the remaining cubic contribution of $Y$ cancels between the layers. A set $C_d$ is a $\tau$-design if
\[
\sum_{x\in C_d} h(x)=0
\]
for every $h\in \operatorname{Harm}_{d}(1)\oplus\dots\oplus\operatorname{Harm}_{d}(\tau)$. Since the original design
\[
X=Y\sqcup(-Y)\subset S^{d-2}
\]
is a spherical $5$-design, every harmonic polynomial of even degree $2$ or $4$
has zero sum already on $Y$. Together with
\[
\sum_{y\in Y}y=0,
\]
this means that $Y$ has harmonic indices $1,2,4$. The final set $C_d$ is antipodal, so all harmonic moments of odd degree vanish
automatically. It therefore remains only to check degrees $2$ and $4$. We impose the standard latitude equations
\[
\frac{1}{5}\sum_{i=1}^5u_i^2=\frac1d,
\qquad
\frac{1}{5}\sum_{i=1}^5u_i^4=\frac{3}{d(d+2)},
\]
as in \cite{ArmanEtAl2026}. Let $\mathcal B_j$ be a basis of $\operatorname{Harm}_{d-1}(j)$. We can use the standard recursive construction of spherical harmonics by
Gegenbauer polynomials \cite{Bajnok1989,HunzikerSepanskiStanke2012}, and writing 
\[
z=\sqrt{1-u_i^2}y, \qquad x=(z,\delta_iu_i)\in\bb R^{d-1}\oplus\bb R, 
\qquad
\|x\|^2=\|z\|^2+u_i^2,
\]
we then evaluate the corresponding basis of $\operatorname{Harm}_d(2)$ on each layer to obtain
\[
u_i^2-\frac{\|x\|^2}{d},
\qquad
\delta_iu_iH_1(z)\quad(H_1\in\mathcal B_1),
\qquad
H_2(z)\quad(H_2\in\mathcal B_2).
\]
The first term vanishes on $C_d$ by the quadratic latitude equation, while the other two vanish because $Y$ has harmonic indices $1$ and $2$. Similarly, for the corresponding  basis of $\operatorname{Harm}_d(4)$ on each layer we obtain 
\begin{gather*}
u_i^4-\frac{6}{d+4}u_i^2\|x\|^2+\frac{3}{(d+2)(d+4)}\|x\|^4,\qquad
\left(\delta_iu_i^3-\frac{3}{d+4}\delta_iu_i\|x\|^2\right)H_1(z),
\ H_1\in\mathcal B_1,\\
\left(u_i^2-\frac{1}{d+4}\|x\|^2\right)H_2(z),
\> H_2\in\mathcal B_2,\qquad
\delta_iu_iH_3(z),
\> H_3\in\mathcal B_3,
\qquad H_4(z), \>H_4\in\mathcal B_4. 
\end{gather*}
The first term vanishes by the two latitude equations, and the terms
containing $H_1$, $H_2$, and $H_4$ vanish because $Y$ has harmonic
indices $1,2,4$. Thus only $\delta_iu_iH_3(z)$ remains. Since $H_3$ is homogeneous of degree $3$,
\[
\delta_i u_iH_3(z)
=
\delta_i u_i(1-u_i^2)^{3/2}H_3(y).
\]
Hence this harmonic moment vanishes for every
$H_3\in\operatorname{Harm}_{d-1}(3)$ provided
\[
\sum_{i=1}^5
\delta_i u_i(1-u_i^2)^{3/2}=0.
\]
Therefore, it remains to find five distinct latitudes satisfying
\[
\sum_{i=1}^5u_i^2=\frac{5}{d},\qquad
\sum_{i=1}^5u_i^4=\frac{15}{d(d+2)},\qquad
\sum_{i=1}^5
\delta_i u_i(1-u_i^2)^{3/2}=0.
\]
We leave the remaining proof in Lemma~\ref{lem:five-design-parameters} which supplies the five distinct latitude parameters and the five signs. The solution is purely technical and finishes the argument for the  existence of the $5$-design with  $60(d^2-1)$ since the construction which is lifted, the half $Y$ has cardinality $6(d^2-1)$ \cite{ArmanEtAl2026}.
\end{proof}

\subsection{Construction of $7$-Designs}

In the case of $\tau=7$ we do not have a convenient starting construction in every even dimension akin to the construction presented in \cite{ArmanEtAl2026} or the SIC-POVMs. We do present a similar result, though, starting from a complex projective design, such as the one, proposed by K{\"o}nig \cite{Konig1999}. We replace the octagon in the circle fiber with a decagon in order to obtain a stronger structure and take the antipodal half $Y\subset S^{d-2}$ which is a $4$-design with a constant $6$ degree moment. We eschew the requirement for all latitudes to be different, deciding to opt for suitable rotations on the same latitude. Thus we place eleven copies of $Y$ on three different latitudes at varying orientations where nine of them are on a single latitude and the remaining two are placed at two distinct latitudes. 

\begin{theorem}
Let $d \geq 9$ be odd, and suppose there exists in  $\mathbb{CP}^{(d-3)/2}$ an equal-weight complex projective 3-design with $M$ points. Then there is a spherical 7-design in $S^{d-1}$ with exactly $110M$ points.
\end{theorem}

\begin{proof}
We label the latitudes with $u_i$ and the rotations 
    \[
    r(\theta_k)=
    \begin{pmatrix}
\cos\theta_k\, I_{\frac{d-1}{2}} & -\sin\theta_k\, I_{\frac{d-1}{2}} \\
\sin\theta_k\, I_{\frac{d-1}{2}} &  \cos\theta_k\, I_{\frac{d-1}{2}} 
\end{pmatrix} \in O(d-1).
    \]
    The chosen angle of rotation is $
    \theta_k=\frac{\pi +kt}{5}$ so as to match the fifth degree harmonic moment of $Y$ where $k=-4,\dots, 4,$ and $t$ determines how far apart the copies are from each other. The points that result are 
\[
\pm(\sqrt{1-u_i^2}yr(\theta_k),u_i),
\]
where the angle of rotation is $0$ for the latitudes $u_2$ and $u_3$.
The set we obtain is determined by the three possible absolute values of the latitude at which $Y$ lies and its nine rotations on the lowest such. With the aim of making the code antipodal, we also include the points of opposite sign to our initial. Consequently, it is implied that the odd moment equations hold and through certain choice of $r(\theta_k)$ at latitude $u_1$ the even moment equations too hold. We proceed to work with the harmonic description of a spherical $7$-design. Naturally, $Y$ satisfies the even moments $2,4$, and $6$, because it follows from the $7$-design
\[
X=Y\sqcup(-Y)\subset S^{d-2}.
\]
Furthermore it also satisfies the harmonic moment equations for $1$ and $3$. The construction consists of disjoint regular pentagons which are spherical $4$-designs (on the circle), so every homogeneous polynomial of degree $1$ or $3$ has zero sum on each pentagon. The antipodal points of $C_d$ leave only the $2,4$ and $6$ degree moments to be checked. We require the latitude equations to hold:
\[
\frac{1}{11}\sum_{i=1}^{11}u_i^2=\frac{1}{d},
\qquad
\frac{1}{11}\sum_{i=1}^{11}u_i^4=\frac{3}{d(d+2)}, \qquad \frac{1}{11}\sum_{i=1}^{11}u_i^6=\frac{15}{d(d+2)(d+4)}.
\]
We proceed with the proof as in the previous section. Again we denote $\mathcal B_j$ to be the basis $\operatorname{Harm}_d(j)$ and choose
\[
x=(z,u_i)\in C_d, \qquad\|x\|^2=\|z\|^2 + u_i^2
\]
for such
\[
z\in \{\sqrt{1-u_i^2}yr(\theta_k), \sqrt{1-u_i^2}y\}\subset \bb{R}^{d-1}.
\]
Through a standard recursive construction of spherical harmonics by Gegenbauer polynomials we evaluate the basis of $\operatorname{Harm}_d(2)$:
\[
u_i^2-\frac{\|x\|^2}{d},
\qquad
u_iH_1(z)\quad(H_1\in\mathcal B_1),
\qquad
H_2(z)\quad(H_2\in\mathcal B_2).
\]
Therefore, for $\operatorname{Harm}_d(4)$
\begin{gather*}
u_i^4-\frac{6}{d+4}u_i^2\|x\|^2+\frac{3}{(d+2)(d+4)}\|x\|^4,\qquad
\left(u_i^3-\frac{3}{d+4}u_i\|x\|^2\right)H_1(z),
\ H_1\in\mathcal B_1,\\
\left(u_i^2-\frac{1}{d+4}\|x\|^2\right)H_2(z),
\> H_2\in\mathcal B_2,\qquad
u_iH_3(z),
\> H_3\in\mathcal B_3,
\qquad H_4(z), \>H_4\in\mathcal B_4. 
\end{gather*}
 Following the same method applied for the $5$-design construction, we obtain the basis of $\operatorname{Harm}_d(6)$,
\begin{align*}
&u_i^6-\frac{15}{d+8}u_i^4\|x\|^2
+\frac{45}{(d+6)(d+8)}u_i^2\|x\|^4
-\frac{15}{(d+4)(d+6)(d+8)}\|x\|^6,\\
&\left(
 u_i^5-\frac{10}{d+8}u_i^3\|x\|^2
+\frac{15}{(d+6)(d+8)} u_i\|x\|^4
\right)H_1(z),
\qquad H_1\in\mathcal B_1,\\
&\left(
u_i^4-\frac{6}{d+8}u_i^2\|x\|^2
+\frac{3}{(d+6)(d+8)}\|x\|^4
\right)H_2(z),
\qquad H_2\in\mathcal B_2,\\
&\left(
 u_i^3-\frac{3}{d+8} u_i\|x\|^2
\right)H_3(z),
\qquad H_3\in\mathcal B_3,\quad
\left(
u_i^2-\frac{1}{d+8}\|x\|^2
\right)H_4(z),
\qquad H_4\in\mathcal B_4,\\
& u_iH_5(z),
\qquad H_5\in\mathcal B_5,
\qquad
H_6(z),
\qquad H_6\in\mathcal B_6.
\end{align*}
The first basis harmonic polynomial vanishes after the evaluation of the even moment latitude equations. The basis is highly dependent on the harmonic moment equations of $Y$. Hence, all terms vanish except the fifth degree ones:
\[
u_i H_5(z)=u_i(1-u_i^2)^{5/2}H_5(r(\theta_k)y).
\]
We again recall that $Y$ is a disjoint union of regular pentagons. Each of them is a $4$-design itself (on the circle) so only the fifth-degree terms remain. Since the nine copies at latitude $u_1$ have rotations $\theta_k$, $k=-4,\ldots,4$, while the copies at $u_2$ and $u_3$ are not rotated, the required equations are
\[
u_1(1-u_1^2)^{5/2}\sum_{k=-4}^{4}\cos(5\theta_k)
+u_2(1-u_2^2)^{5/2}
+u_3(1-u_3^2)^{5/2}=0,
\]
\[
u_1(1-u_1^2)^{5/2}\sum_{k=-4}^{4}\sin(5\theta_k)=0.
\]
Since $\sin$ is an odd function, the latter follows immediately. Once again we leave the solution of the required system to a lemma in the technical section. Lemma~\ref{lem:seven-design-parameters} provides the three latitude parameters and rotations that give a solution and proves that the construction is a valid $7$-design with exactly $22|Y|=110M$ points.
\end{proof}

\section{Technical lemmas}

\begin{lemma}\label{lem:five-design-parameters}
For every odd integer $d\geq25$, there exist five distinct latitude parameters
$0<u_1<u_2<u_3<u_4<u_5<1$ with a sign pattern
$(1,1,1,1,-1)$, such that
\[
\sum_{i=1}^5u_i^2=\frac{5}{d},\qquad
\sum_{i=1}^5u_i^4=\frac{15}{d(d+2)},\qquad
\sum_{i=1}^5\delta_i u_i(1-u_i^2)^{3/2}=0.
\]
\end{lemma}
\begin{proof}
For convenience, we set $v_i=\sqrt{d-1}\,u_i$. Then the equations become
\[
\sum_{i=1}^5v_i^2=A_d\coloneq\frac{5(d-1)}{d},\qquad
\sum_{i=1}^5v_i^4=B_d\coloneq\frac{15(d-1)^2}{d(d+2)},\qquad
\sum_{i=1}^5
\delta_i v_i
\left(1-\frac{v_i^2}{d-1}\right)^{3/2}=0.
\]
We show that a solution exists for odd $d\geq 25$. Let $\epsilon=1/100$, then look for a solution of the form
\[
v_1=a-\epsilon, \qquad
v_2=a,\qquad
v_3=a+\epsilon,\qquad
v_4=b,\qquad
v_5=c,
\]
where $a\in [1/50,7/20]$ and $a+\epsilon<b<c$. We set
\[
s(a):=b^2+c^2=A_d-3a^2-2\epsilon^2, \qquad
p(a):=b^2c^2=\frac{s(a)^2-\left(B_d-3a^4-12a^2\epsilon^2-2\epsilon^4\right)}{2}.
\]
Now by Vieta's formulas, $b^2$ and $c^2$ are the roots of a quadratic equation, denoted with $f$, which means that if 
\[
D(a)=s(a)^2-4p(a),
\]
then
\[
b^2=\frac{s(a)-\sqrt{D(a)}}{2},
\qquad
c^2=\frac{s(a)+\sqrt{D(a)}}{2}.
\]
The following bounds prove to be enough for further operations:
\[
5>s(a)>4, \qquad D(a)>\frac{64}{25}, \qquad p(a)>\frac{13}{4}.
\]
The first comes immediately from $d\geq25$, $7/20\geq a$, and
$\epsilon=1/100$:
\[
s(a)\geq \frac{24}{5}-\frac{147}{400}-\frac{1}{5000}>4, \qquad 5>A_d>s(a).
\]
$D(a)$ is strictly increasing on the interval of $a$ because 
\[
\frac{dD}{da}= 12a\left(A_d-5a^2-6\epsilon^2\right)>0.
\]
Hence its minimum occurs at $a=2\epsilon=1/50$, where
\[
D(2\epsilon)> 2B_d-A_d^2 = \frac{5(d-10)(d-1)^2}{d^2(d+2)} \geq \frac{64}{25}.
\]
This proves the second bound on the discriminant. Similarly,
\[
0>\frac{dp}{da} = 6a\left(4a^2+4\epsilon^2-A_d\right),
\]
so we need to estimate $p(7/20)$. In order to obtain a lower bound for $p(a)$ we take it as a limit for a function of $d$ 
\[
\lim _{d \to \infty} p_a(d)=5-15a^2+6a^4 -10\epsilon^2+12a^2\epsilon^2+3\epsilon^4.
\]
It follows that $p(a)$ is greater than the limit, so by substituting $a=7/20$, we obtain the third inequality:
\[
p\left(\frac{7}{20}\right)> 5-\frac{147}{80}  +\frac{7203}{80000} - \frac{1}{1000} > \frac{13}{4}.
\]

Since
\[
f\left(\frac14\right) > \frac{1}{16}-\frac54+\frac{13}{4} = \frac{33}{16}>0
\]
and the minimum of the function lies at $s(a)/2>2$, it follows that $b>1/2$. Now from 
\[
\frac{1}{2} >\frac{7}{20}+\frac{1}{100} = \frac{9}{25}\geq a+\epsilon ,
\]
we conclude that the following inequalities are satisfied
\[
c>b>a+\epsilon>a>a-\epsilon>0.
\]

Finally, 
\[
d-1>5>s(a)>c^2 \implies \sqrt{d-1}>c, 
\]
which leaves the third equation to be satisfied. Let
\[
g(t) = t\left(1-\frac{t^2}{d-1}\right)^{3/2}, \qquad
F_d(a) = g(a-\epsilon)+g(a)+g(a+\epsilon)+g(b)-g(c).
\]
We note that $F_d$ is continuous on the interval. At its lower endpoint, $a=1/50=2\epsilon$, we have
\[
\frac{3}{50}=6\epsilon>g(\epsilon)+g(2\epsilon)+g(3\epsilon).
\]
On the other hand,
\[
c-b = \frac{\sqrt{D}}{b+c} > \frac{4\sqrt{10}}{25} > \frac{1}{2}.
\]

For $c\geq t\geq b$ we have $t^2<5$, and hence
\[
\frac{dg}{dt}
=
\sqrt{1-\frac{t^2}{d-1}}
\left(1-\frac{4t^2}{d-1}\right)
>
\frac{7}{48}.
\]
From these two inequalities and the mean value theorem we get
\[
g(c)-g(b)>\frac{7}{48}\cdot\frac{1}{2}>\frac{3}{50},
\]
hence
\[
0>F_d\left(\frac{1}{50}\right).
\]
At the upper endpoint $a=7/20$, the first three arguments of $g$ are
\[
\frac{17}{50}, \qquad \frac{7}{20}, \qquad \frac{9}{25},
\]
whose sum is $21/20$. We apply Bernoulli's inequality:
\[
\left(1-\frac{t^2}{d-1}\right)^{3/2} \geq 1-\frac{3t^2}{2(d-1)} > \frac{99}{100}
\]
for $9/25\geq t>0$. Consequently, 
\[
g\left(\frac{17}{50}\right) + g\left(\frac{7}{20}\right) + g\left(\frac{9}{25}\right) > \frac{99}{100}\cdot\frac{21}{20} = \frac{2079}{2000}.
\]

Using $5>A_d$ and $p>13/4$, we obtain
\[
s\left(\frac{7}{20}\right) < 5-3\left(\frac{7}{20}\right)^2 = \frac{1853}{400},
\qquad
2bc=2\sqrt{p}>\sqrt{13}>\frac{18}{5}.
\]
Therefore
\[
\left(\frac{51}{50}\right)^2>\frac{1853}{400}-\frac{18}{5}
=
\frac{413}{400}>(c-b)^2
=
s-2bc \implies \frac{51}{50}>c-b.
\]

Now again from the mean value theorem and  $1>\dfrac{dg}{dt}>0$ on $[b,c]$,
\[
\frac{2079}{2000}>\frac{51}{50}>g(c)-g(b).
\]
Thus
\[
F_d\left(\frac{7}{20}\right)>0.
\]
By the intermediate value theorem, there exists
$a_0\in \left(1/50,7/20\right)$ such that $F_d(a_0)=0.$ Now we can take $a_0-\epsilon, a_0, a_0+\epsilon, b_0, c_0$ which satisfy all three required equations and normalize them with $\sqrt{d-1}$, which gives the required five distinct latitude parameters.
\end{proof}

\begin{lemma}\label{lem:seven-design-parameters}
For every odd integer $d\geq9$, there exist latitude parameters
$0<u_1<u_2<u_3<1$ and $t\in(0,2\pi/9)$ such that
\begin{align*}
&9u_1^2+u_2^2+u_3^2=\frac{11}{d},\qquad
9u_1^4+u_2^4+u_3^4=\frac{33}{d(d+2)},\qquad
9u_1^6+u_2^6+u_3^6=\frac{165}{d(d+2)(d+4)},\\
&u_1(1-u_1^2)^{5/2}\sum_{k=-4}^{4}\cos(\pi+kt)
+u_2(1-u_2^2)^{5/2}+u_3(1-u_3^2)^{5/2}=0.
\end{align*}
\end{lemma}
\begin{proof}
In order to establish the existence we rescale the latitudes
\[
a=\sqrt{d}u_1,\qquad b=\sqrt{d}u_2, \qquad c=\sqrt{d}u_3,
\]
and the system becomes
\[
9a^2+b^2+c^2=11,\qquad9a^4+b^4+c^4=B_d\coloneq\frac{33d^2}{d(d+2)},\qquad9a^6+b^6+c^6=C_d\coloneq\frac{165d^3}{d(d+2)(d+4)}.
\]
From here on we treat $b^2$ and $c^2$ as roots of some quadratic equation with a discriminant
\[
D=2(B_d-9a^4)-(11-9a^2)^2.
\]
When expressed through $a$ and $d$, the two latitudes become
\[
b=\sqrt{\frac{11-9a^2-\sqrt{D}}{2}}, \qquad c=\sqrt{\frac{11-9a^2+\sqrt{D}}{2}}.
\]
Therefore, the final equation can be expressed as follows
\[
F_d(a)=9a^6+\left(\frac{11-9a^2-\sqrt{D}}{2}\right)^3+ \left(\frac{11-9a^2+\sqrt{D}}{2}\right)^3-C_d=0
\]
and it is enough to prove that there is an $a$ that satisfies it, in order to obtain the necessary result. Choose $a\in [0.66,0.74]$ and $d\geq9$, on this range the discriminant is positive, and we are able to verify the existence of a root. For this, it is sufficient to prove that 
\[
F_d(0.74)>0>F_d(0.66). 
\]
First, we prove that the left-hand side holds. It is equal to
\[
F_d(0.74)=11\left(-45\left(\frac{1131}{2500} \right)^3+27\left(\frac{1131}{2500} \right)\frac{d-1}{d+2}-8\frac{(d-1)(d-2)}{(d+2)(d+4)}\right),
\]
and to determine its growth for $d$ we denote it as a function of $d$ and its derivative
\[
F_d(0.74)=F_{0.74}(d), \qquad 0>\frac{dF_{0.74}}{dd}=\frac{11(-88389d^2+492888d+2185776)}{2500(d+2)^2(d+4)^2}
\] 
for $d\geq 9$. As the function decreases we can set a lower bound through its value for $d$ to infinity
\[
F_{0.74}(d)>\lim_{d\to \infty}F_{0.74}(d)>0.53.
\]
Let us check now the right hand side. The function evaluates as follows   
\[
F_{0.66}(d)=11\left(-45\left(\frac{1411}{2500} \right)^3+27\left( \frac{1411}{2500}\right)\frac{d-1}{d+2}-8\frac{(d-1)(d-2)}{(d+2)(d+4)} \right).
\]
Since we only consider odd dimensions, we can compare consecutive values directly:
\[
F_{0.66}(d+2)-F_{0.66}(d) = -\frac{33(21903d-268582)}{1250(d+2)(d+4)(d+6)}.
\]
This is positive for $d=9,11$ and negative for every odd $d \geq 13$, so
$F_{0.66}$ attains its maximum at $d=13$. Furthermore, $0>-0.44>F_{0.66}(13)>-0.55,$
and thus $0>F_{0.66}(13)\geq F_{0.66}(d).$

This concludes the required result, and proves the existence of latitudes that satisfy the moment equations. The next result that has to be established is the existence of the nine rotation angles, which we establish through $t$, that satisfies
\[
Q(t)=a\left(1-\frac{a^2}{d}\right)^{5/2}\sum_{k=-4} ^4 \cos(\pi+kt)+b\left(1-\frac{b^2}{d}\right)^{5/2}+c\left(1-\frac{c^2}{d}\right)^{5/2}=0.
\]
We choose $t\in\left[0,\frac{2\pi}{9}\right]$ and following an analogous method prove that at the endpoints of the interval the function changes signs and consequently there is a value of $t$ for which the equation holds. It is trivial at the upper endpoint as 
\[
\sum_{k=-4} ^4 \cos\left(\pi+\frac{2\pi k}{9}\right)=0,\qquad
Q\left(\frac{2\pi}{9} \right)=b\left(1-\frac{b^2}{d}\right)^{5/2}+c\left(1-\frac{c^2}{d}\right)^{5/2}>0.
\]
This establishes the left hand side of 
\[
Q\left(\frac{2\pi}{9}\right)>0>Q\left(0\right).
\]
The right hand side is equivalent to proving
\[
0>Q\left(0\right)=-9a\left(1-\frac{a^2}{d} \right)^{5/2} + b \left( 1-\frac{b^2}{d}\right)^{5/2} + c\left(1-\frac{c^2}{d} \right)^{5/2}.
\]
or to be more straightforward
\[
9a\left(1-\frac{a^2}{d} \right)^{5/2}>b \left( 1-\frac{b^2}{d}\right)^{5/2} + c\left(1-\frac{c^2}{d} \right)^{5/2}.
\]
For $d \geq 9$ and in the chosen interval, we can check that $\sqrt{d}>c>b>a>0$. Now in order to establish this inequality we also prove that 
\[
9a>b+c.
\]
By applying the Cauchy-Schwarz inequality and $a\geq0.66$, we obtain
\[
81a^2>2(11-9a^2)=2(b^2+c^2), \quad2(b^2+c^2)>(b+c)^2\implies 81a^2>(b+c)^2
\] 
and due to the latitudes being positive it is exactly what was required. Now we perform the following multiplication, and the result for the function $Q\left(0\right)$ follows fairly simply:
\begin{align*}
&9a\left(1-\frac{a^2}{d} \right)^{5/2} >b\left(1-\frac{a^2}{d} \right)^{5/2} +c\left(1-\frac{a^2}{d} \right)^{5/2}\\
&\implies9a\left(1-\frac{a^2}{d} \right)^{5/2}>b \left( 1-\frac{b^2}{d}\right)^{5/2} + c\left(1-\frac{c^2}{d} \right)^{5/2}.
\end{align*}
This concludes the proof for existence of such $t$ as
\[
Q\left(\frac{2\pi}{9}\right)>0>Q\left(0\right)
\]
and in combination with the existence of latitudes for which the moment equations are satisfied, this proves the lemma.
\end{proof}

\section{Specific constructions}

The technical lemmas prove results that hold asymptotically, however, analogical results can be proven for a smaller number of layers (and in exactly the same way). For four layers, the system has a valid solution only up to and including dimension $23$. With three layers, the method is available up to dimension $7$. One can also ask if we can allow for arbitrary rotations and multiple copies on the same latitude, as in the construction for the $7$-designs. As in the $7$-design construction, it provides a valid construction up to dimension $41$ with nine copies of the half-design. Thus, in the following table we give the best results known to us up to dimension $29$.

\begin{table}[h!]
\label{table:5-design}
\centering
\begin{tabular}{r r r l}
\toprule
$d$ & DGS & $|C_d|$ & Source \\
\midrule
 7 &   56 &    56 &      \cite{CohnKumar2007} \\
 9 &   90 &    384 &     \cite{ScottGrassl2010,Lindblad2026} \\
11 &  132 &    600 &    \cite{ScottGrassl2010,Lindblad2026} \\
13 &  182 &    864 &    \cite{ScottGrassl2010,Lindblad2026} \\
15 &  240 &  1\,176 &   \cite{ScottGrassl2010,Lindblad2026} \\
17 &  306 &  1\,152 &   \cite{SudaXuYu2025} \\
19 &  380 &  1\,944 &   \cite{ScottGrassl2010,Lindblad2026} \\
21 &  462 &  2\,400 &   \cite{ScottGrassl2010,Lindblad2026} \\
23 &  552 &   552 &   \cite{CohnKumar2007} \\
25 &  650 &  4\,320 &   \cite{ScottGrassl2010,Lindblad2026} \\
27 &  756 &  5\,070 &   \cite{ScottGrassl2010,Lindblad2026} \\
29 &  870 &  5\,880 &   \cite{ScottGrassl2010,Lindblad2026} \\
\bottomrule
\end{tabular}
\caption{Cardinalities of $5$-design constructions in dimension $d$ and their corresponding spherical
$5$-designs in $\bb{R}^{d-1}$.}
\label{tab:odd-five-designs}
\end{table}
\vspace{6em}
If Zauner's hypothesis is true, we do not need to rely on the numerical results of Scott and Grassl. It is also noteworthy that $E_8$ and the Leech lattice provide derived codes for $d=7$ and $d=23$ with cardinalities $56$ and $552$, respectively, which are tight \cite{CohnKumar2007}. For $d=17$ the construction comes from a maximal set of mutually unbiased bases.

For $7$-designs, we can also use fewer than $10$ layers for lower dimensions. We find that the following numbers of layers provide a design up to the indicated dimension.
\begin{table}[b]
\centering
\begin{tabular}{c|ccccc}
\hline
\text{layers} & 10 & 9 & 8 & 7 & 6 \\
\hline
$\max d$ & 75 & 31 & 19 & 11 & 9 \\
\hline
\end{tabular}
\caption{Maximum odd dimension for each number of layers for $7$-designs.}
\label{tab:layer-dimension}
\end{table}

Except in $d=3$ and $d=23$, the best cardinalities to our knowledge come from the lifts we describe here. For $d=23$ the DGS bound is reached by a derived code of the Leech lattice.

\begin{table}[b]
\centering
\begin{tabular}{r r r l}
\toprule
$d$ & DGS & $|C_d|$ & Source \\
\midrule
 3 &      20 &              24 & \cite{HardinSloane1996}\\
 5 &      70 &             480 & \cite{HarpePacheVenkov2007}\\
 7 &     168 &           5\,400 & \cite{Misawa2026,IosueEtAl2024}\\
 9 &     330 &           2\,400 & \cite{Hoggar1982}\\
11 &     572 &          70\,400 &
    \cite{Hoggar1982,LyubichShatalova2013,RabauBajnok1991}\\
13 &     910 &          10\,080 & \cite{Hoggar1982}\\
15 &  1\,360 &          24\,192 & \cite{HarpePacheVenkov2007}\\
17 &  1\,938 &          69\,120 & \cite{HarpePacheVenkov2007}\\
19 &  2\,660 &       1\,728\,000 &
    \cite{KuengGross2015,LyubichShatalova2013,RabauBajnok1991}\\
21 &  3\,542 &         142\,560 & \cite{HarpePacheVenkov2007}\\
23 &  4\,600 &           4\,600 & \cite{CohnKumar2007}\\
25 &  5\,850 &         142\,560 & \cite{Lindblad2026}\\
27 &  7\,308 &      786\, 240 &
    \cite{BachocVenkov2001}\\
29 &  8\,990 &      77\,968\,800 &
    \cite{Misawa2026,IosueEtAl2024}\\
\bottomrule
\end{tabular}
\caption{Cardinalities of the $7$-design construction and their respective complex projective $3$-designs
$\mathbb{CP}^{(d-3)/2}$.}
\label{tab:odd-seven-designs}
\end{table}
\vspace{3em}
The sources of the lifted constructions, however, are noticeably diverse. Small dimensions are mostly covered by exceptional configurations. Another exceptional configuration exists in dimension $12$ due to Nasmith \cite{Nasmith2022}. Several highly symmetric lattice shells and their unions \cite{HarpePacheVenkov2007} also provide strong constructions in several dimensions. In the remaining dimensions under consideration, we either have no relatively satisfactory construction or our constructions start from lifting complex projective $3$-designs. These projective complex designs arise either sporadically or from recursive constructions.

\section{Acknowledgements}

The research of the first author is supported by Bulgarian NSF grant KP-06-
N72/6-2023. The research of the second author is supported by Bulgarian National Program "Education with Science". We sincerely thank to Peter Boyvalenkov for the valuable advice and discussions.

\section{AI Use Statement}

During the work on the article, Artificial Intelligence was used for assistance in the numerical exploration of the feasibility of different lifting structures (number of layers, orientations, etc.), as well as the bibliographic search. AI has had no input on research hypotheses and core methods. AI input in proofs is confined to numeric simplifications. All AI output was verified by the authors.

Artificial intelligence was also used for proofreading, language corrections and formatting. 

The authors assume full responsibility for the content of the article.

\nocite{*}
\bibliographystyle{plain}
\bibliography{refs}

@misc{ArmanEtAl2026,
  author        = {Arman, Andrii and Bondarenko, Andriy and
                   Prymak, Andriy and Radchenko, Danylo},
  title         = {A Construction of Spherical {5}-Designs with
                   {$O(d^2)$} Points},
  year          = {2026},
  eprint        = {2606.01376},
  archivePrefix = {arXiv},
  primaryClass  = {math.CO},
  doi           = {10.48550/arXiv.2606.01376},
  url           = {https://arxiv.org/abs/2606.01376}
}

@phdthesis{Bajnok1989,
  author  = {Bajnok, B{\'e}la},
  title   = {Construction of Spherical {$t$}-Designs},
  school  = {The Ohio State University},
  address = {Columbus, Ohio},
  year    = {1989},
  type    = {Ph.D. dissertation},
  url     = {https://etd.ohiolink.edu/acprod/odb_etd/etd/r/1501/10?clear=10&p10_accession_num=osu1487671108305599}
}

@article{HunzikerSepanskiStanke2012,
  author  = {Hunziker, Markus and Sepanski, Mark R. and Stanke, Ronald J.},
  title   = {The Minimal Representation of the Conformal Group and
             Classical Solutions to the Wave Equation},
  journal = {Journal of Lie Theory},
  volume  = {22},
  number  = {2},
  pages   = {301--360},
  year    = {2012},
  doi     = {10.5802/jolt.671},
  url     = {https://jolt.centre-mersenne.org/item/10.5802/jolt.671/}
}

@article{DelsarteGoethalsSeidel1977,
  author  = {Delsarte, Philippe and Goethals, Jean-Marie and Seidel, Johan J.},
  title   = {Spherical Codes and Designs},
  journal = {Geometriae Dedicata},
  volume  = {6},
  number  = {3},
  pages   = {363--388},
  year    = {1977},
  doi     = {10.1007/BF03187604},
  url     = {https://doi.org/10.1007/BF03187604}
}

@article{BannaiDamerell1979,
  author  = {Bannai, Eiichi and Damerell, Robert M.},
  title   = {Tight Spherical Designs, {I}},
  journal = {Journal of the Mathematical Society of Japan},
  volume  = {31},
  number  = {1},
  pages   = {199--207},
  year    = {1979},
  doi     = {10.2969/jmsj/03110199},
  url     = {https://doi.org/10.2969/jmsj/03110199}
}

@article{BannaiDamerell1980,
  author  = {Bannai, Eiichi and Damerell, Robert M.},
  title   = {Tight Spherical Designs, {II}},
  journal = {Journal of the London Mathematical Society},
  series  = {2},
  volume  = {21},
  number  = {1},
  pages   = {13--30},
  year    = {1980},
  doi     = {10.1112/jlms/s2-21.1.13},
  url     = {https://doi.org/10.1112/jlms/s2-21.1.13}
}

@article{BoyvalenkovNikova1998,
  author  = {Boyvalenkov, Peter G. and Nikova, Svetla},
  title   = {Improvements of the Lower Bounds on the Size of Some Spherical Designs},
  journal = {Mathematica Balkanica},
  series  = {New Series},
  volume  = {12},
  number  = {1--2},
  pages   = {151--160},
  year    = {1998},
  url     = {https://www.math.bas.bg/infres/MathBalk/MB-12/MB-12-151-160.pdf}
}

@article{BoyvalenkovDanevNikova1999,
  author  = {Boyvalenkov, Peter G. and Danev, Danyo and Nikova, Svetla},
  title   = {Nonexistence of Certain Spherical Designs of Odd Strengths and Cardinalities},
  journal = {Discrete \& Computational Geometry},
  volume  = {21},
  pages   = {143--156},
  year    = {1999},
  doi     = {10.1007/PL00009406},
  url     = {https://doi.org/10.1007/PL00009406}
}

@article{BoyvalenkovStoyanova2010,
  author  = {Boyvalenkov, Peter G. and Stoyanova, Maya},
  title   = {A New Asymptotic Bound of the Minimum Possible Odd Cardinality of Spherical {$(2k-1)$}-Designs},
  journal = {Discrete Mathematics},
  volume  = {310},
  number  = {15--16},
  pages   = {2170--2175},
  year    = {2010},
  doi     = {10.1016/j.disc.2010.04.007},
  url     = {https://doi.org/10.1016/j.disc.2010.04.007}
}

@article{Bajnok1991FourFive,
  author  = {Bajnok, B{\'e}la},
  title   = {Construction of Spherical {4}- and {5}-Designs},
  journal = {Graphs and Combinatorics},
  year    = {1992},
  doi     = {10.1007/BF01787629},
  url     = {https://doi.org/10.1007/BF01787629}
}

@article{Bajnok1992TDesigns,
  author  = {Bajnok, B{\'e}la},
  title   = {Construction of Spherical {$t$}-Designs},
  journal = {Geometriae Dedicata},
  volume  = {43},
  number  = {2},
  pages   = {167--179},
  year    = {1992},
  doi     = {10.1007/BF00147866},
  url     = {https://doi.org/10.1007/BF00147866}
}

@article{Reznick1995,
  author  = {Reznick, Bruce},
  title   = {Some Constructions of Spherical {5}-Designs},
  journal = {Linear Algebra and its Applications},
  volume  = {226--228},
  pages   = {163--196},
  year    = {1995},
  doi     = {10.1016/0024-3795(95)00101-V},
  url     = {https://doi.org/10.1016/0024-3795(95)00101-V}
}

@article{Kuperberg2006Codes,
  author  = {Kuperberg, Greg},
  title   = {Numerical Cubature Using Error-Correcting Codes},
  journal = {SIAM Journal on Numerical Analysis},
  volume  = {44},
  number  = {3},
  pages   = {897--907},
  year    = {2006},
  doi     = {10.1137/040615572},
  url     = {https://doi.org/10.1137/040615572}
}

@article{Kuperberg2006HatBox,
  author  = {Kuperberg, Greg},
  title   = {Numerical Cubature from Archimedes' Hat-Box Theorem},
  journal = {SIAM Journal on Numerical Analysis},
  volume  = {44},
  number  = {3},
  pages   = {908--935},
  year    = {2006},
  doi     = {10.1137/040615584},
  url     = {https://doi.org/10.1137/040615584}
}

@article{MohammadpourWaldron2024,
  author  = {Mohammadpour, Mozhgan and Waldron, Shayne},
  title   = {Complex Spherical Designs from Group Orbits},
  journal = {Journal of Approximation Theory},
  volume  = {301},
  pages   = {106047},
  year    = {2024},
  doi     = {10.1016/j.jat.2024.106047},
  url     = {https://doi.org/10.1016/j.jat.2024.106047}
}

@misc{Misawa2026,
  author        = {Misawa, Ryutaro},
  title         = {Explicit Construction of Spherical {5}- and {7}-Designs},
  year          = {2026},
  eprint        = {2602.19757},
  archivePrefix = {arXiv},
  primaryClass  = {math.CO},
  doi           = {10.48550/arXiv.2602.19757},
  url           = {https://arxiv.org/abs/2602.19757}
}

@article{BannaiZhaoEtAl2018,
  author  = {Bannai, Eiichi and Zhao, Da and Zhu, Lin and
             Zhu, Yan and Zhu, Yinfeng},
  title   = {Half of an Antipodal Spherical Design},
  journal = {Archiv der Mathematik},
  volume  = {110},
  number  = {5},
  pages   = {459--466},
  year    = {2018},
  doi     = {10.1007/s00013-017-1141-1},
  url     = {https://doi.org/10.1007/s00013-017-1141-1}
}

@article{ZhuBannaiEtAl2017,
  author  = {Zhu, Yan and Bannai, Eiichi and Bannai, Etsuko and
             Kim, Kyoung-Tark and Yu, Wei-Hsuan},
  title   = {On Spherical Designs of Some Harmonic Indices},
  journal = {The Electronic Journal of Combinatorics},
  volume  = {24},
  number  = {2},
  pages   = {P2.14},
  year    = {2017},
  doi     = {10.37236/6437},
  url     = {https://doi.org/10.37236/6437}
}

@incollection{Konig1999,
  author    = {K{\"o}nig, Hermann},
  title     = {Cubature Formulas on Spheres},
  booktitle = {Advances in Multivariate Approximation},
  editor    = {Hau{\ss}mann, Wolfgang and Jetter, Kurt and Reimer, Manfred},
  series    = {Mathematical Research},
  volume    = {107},
  pages     = {201--211},
  publisher = {Wiley-VCH},
  address   = {Berlin},
  year      = {1999},
  note      = {Proceedings of the conference held in Witten-Bommerholz, 1998}
}

@article{Hoggar1982,
  author  = {Hoggar, Stuart G.},
  title   = {{$t$}-Designs in Projective Spaces},
  journal = {European Journal of Combinatorics},
  volume  = {3},
  number  = {3},
  pages   = {233--254},
  year    = {1982},
  doi     = {10.1016/S0195-6698(82)80035-8},
  url     = {https://doi.org/10.1016/S0195-6698(82)80035-8}
}

@misc{KuengGross2015,
  author        = {Kueng, Richard and Gross, David},
  title         = {Qubit Stabilizer States Are Complex Projective {$3$}-Designs},
  year          = {2015},
  eprint        = {1510.02767},
  archivePrefix = {arXiv},
  primaryClass  = {quant-ph},
  doi           = {10.48550/arXiv.1510.02767},
  url           = {https://arxiv.org/abs/1510.02767}
}

@article{IosueEtAl2024,
  author  = {Iosue, Joseph T. and Mooney, T. C. and
             Ehrenberg, Adam and Gorshkov, Alexey V.},
  title   = {Projective Toric Designs, Quantum State Designs,
             and Mutually Unbiased Bases},
  journal = {Quantum},
  volume  = {8},
  pages   = {1546},
  year    = {2024},
  doi     = {10.22331/q-2024-12-03-1546},
  url     = {https://doi.org/10.22331/q-2024-12-03-1546}
}

@misc{LyubichShatalova2013,
  author        = {Lyubich, Yuri I. and Shatalova, Oksana A.},
  title         = {A Recursive Construction of Projective Cubature Formulas
                   and Related Isometric Embeddings},
  year          = {2013},
  eprint        = {1310.4562},
  archivePrefix = {arXiv},
  primaryClass  = {math.FA},
  url           = {https://arxiv.org/abs/1310.4562}
}

@article{Kuperberg2006,
  author  = {Kuperberg, Greg},
  title   = {Numerical Cubature from {Archimedes'} Hat-Box Theorem},
  journal = {SIAM Journal on Numerical Analysis},
  volume  = {44},
  number  = {3},
  pages   = {908--935},
  year    = {2006},
  doi     = {10.1137/040615584},
  url     = {https://doi.org/10.1137/040615584}
}

@article{Lindblad2026,
  author  = {Ayodeji Lindblad},
  title   = {Designs Related Through Projective and Hopf Maps},
  journal = {Discrete \& Computational Geometry},
  volume  = {75},
  pages   = {950--968},
  year    = {2026},
  doi     = {10.1007/s00454-025-00805-7},
  url     = {https://link.springer.com/article/10.1007/s00454-025-00805-7}
}

@article{ScottGrassl2010,
  author  = {A. J. Scott and M. Grassl},
  title   = {{SIC-POVMs}: A New Computer Study},
  journal = {Journal of Mathematical Physics},
  volume  = {51},
  pages   = {042203},
  year    = {2010},
  doi     = {10.1063/1.3374022},
  url     = {https://arxiv.org/abs/0910.5784}
}

@inproceedings{KlappeneckerRoetteler2005,
  author    = {Andreas Klappenecker and Martin R{\"o}tteler},
  title     = {Mutually Unbiased Bases are Complex Projective 2-Designs},
  booktitle = {Proceedings of the 2005 IEEE International Symposium on Information Theory},
  pages     = {1740--1744},
  year      = {2005},
  doi       = {10.1109/ISIT.2005.1523643},
  url       = {https://arxiv.org/abs/quant-ph/0502031}
}

@misc{SudaXuYu2025,
  author        = {Sho Suda and Zili Xu and Wei-Hsuan Yu},
  title         = {Existence and Nonexistence of Spherical $5$-Designs of Minimal Type},
  year          = {2025},
  eprint        = {2508.18685},
  archivePrefix = {arXiv},
  primaryClass  = {math.CO},
  url           = {https://arxiv.org/abs/2508.18685}
}

@article{MartinMuzychukWilliford2007,
  author  = {William J. Martin and Mikhail Muzychuk and Jason Williford},
  title   = {Imprimitive Cometric Association Schemes: Constructions and Analysis},
  journal = {Journal of Algebraic Combinatorics},
  volume  = {25},
  pages   = {399--415},
  year    = {2007},
  doi     = {10.1007/s10801-006-0043-2},
  url     = {https://emis.de/ft/43670}
}

@article{HughesWaldron2021,
  author  = {Daniel Hughes and Shayne Waldron},
  title   = {Spherical $(t,t)$-Designs with a Small Number of Vectors},
  journal = {Linear Algebra and its Applications},
  volume  = {608},
  pages   = {84--106},
  year    = {2021},
  doi     = {10.1016/j.laa.2020.08.010},
  url     = {https://www.math.auckland.ac.nz/~waldron/Preprints/Numerical-t-designs/1-s2.0-S0024379520303839-main.pdf}
}

@article{HardinSloane1996,
  author  = {Hardin, Ronald H. and Sloane, Neil J. A.},
  title   = {McLaren's Improved Snub Cube and Other New Spherical Designs in Three Dimensions},
  journal = {Discrete \& Computational Geometry},
  volume  = {15},
  number  = {4},
  pages   = {429--441},
  year    = {1996},
  doi     = {10.1007/BF02711518},
  url     = {https://arxiv.org/abs/math/0207211}
}

@article{CohnConwayElkiesKumar2007,
  author  = {Cohn, Henry and Conway, John H. and Elkies, Noam D. and Kumar, Abhinav},
  title   = {The {$D_4$} Root System Is Not Universally Optimal},
  journal = {Experimental Mathematics},
  volume  = {16},
  number  = {3},
  pages   = {313--320},
  year    = {2007},
  doi     = {10.1080/10586458.2007.10129008},
  url     = {https://arxiv.org/abs/math/0607447}
}

@article{BoyvalenkovEtAl2025,
  author  = {Boyvalenkov, Peter G. and Dragnev, Peter D. and Hardin, Douglas P. and Saff, Edward B. and Stoyanova, Maya M.},
  title   = {Universal Minima of Discrete Potentials for Sharp Spherical Codes},
  journal = {Revista Matem{\'a}tica Iberoamericana},
  volume  = {41},
  number  = {2},
  pages   = {603--650},
  year    = {2025},
  doi     = {10.4171/RMI/1509},
  url     = {https://arxiv.org/abs/2211.00092},
}

@article{BoseChowla1962,
  author  = {Bose, R. C. and Chowla, S.},
  title   = {Theorems in the additive theory of numbers},
  journal = {Commentarii Mathematici Helvetici},
  volume  = {37},
  pages   = {141--147},
  year    = {1962},
  doi     = {10.1007/BF02566968},
  url     = {https://doi.org/10.1007/BF02566968}
}

@misc{ZhuMaoYi2024,
  author        = {Zhu, Huangjun and Mao, Chengsi and Yi, Changhao},
  title         = {Third moments of qudit {Clifford} orbits and 3-designs based on magic orbits},
  year          = {2024},
  eprint        = {2410.13575},
  archivePrefix = {arXiv},
  primaryClass  = {quant-ph},
  doi           = {10.48550/arXiv.2410.13575},
  url           = {https://arxiv.org/abs/2410.13575}
}

@article{Xiang2022,
  author  = {Xiang, Ziqing},
  title   = {Explicit spherical designs},
  journal = {Algebraic Combinatorics},
  volume  = {5},
  number  = {2},
  pages   = {347--369},
  year    = {2022},
  doi     = {10.5802/alco.213},
  url     = {https://alco.centre-mersenne.org/articles/10.5802/alco.213/}
}

@incollection{GoethalsSeidel1981,
  author    = {Goethals, J. M. and Seidel, J. J.},
  title     = {Cubature Formulae, Polytopes, and Spherical Designs},
  booktitle = {The Geometric Vein: The Coxeter Festschrift},
  editor    = {Davis, Chandler and Gr{\"u}nbaum, Branko and Sherk, F. A.},
  publisher = {Springer},
  address   = {New York},
  pages     = {203--218},
  year      = {1981},
  doi       = {10.1007/978-1-4612-5648-9_13},
  url       = {https://doi.org/10.1007/978-1-4612-5648-9_13}
}

@article{BaladramSuprijanto2014,
  author  = {Baladram, Mohammad Samy and Suprijanto, Djoko},
  title   = {New spherical {$(2s+1)$}-designs from {Kuperberg}'s set:
             An experimental result},
  journal = {Applied Mathematics and Computation},
  volume  = {249},
  pages   = {45--52},
  year    = {2014},
  doi     = {10.1016/j.amc.2014.10.045},
  url     = {https://doi.org/10.1016/j.amc.2014.10.045}
}

@article{SamuelGedik2024,
  author  = {Solomon B. Samuel and Zafer Gedik},
  title   = {Group theoretical classification of {SIC-POVMs}},
  journal = {Journal of Physics A: Mathematical and Theoretical},
  volume  = {57},
  number  = {29},
  pages   = {295304},
  year    = {2024},
  doi     = {10.1088/1751-8121/ad5ca9},
  url     = {https://doi.org/10.1088/1751-8121/ad5ca9}
}

@article{RabauBajnok1991,
  author  = {Patrick Rabau and B{\'e}la Bajnok},
  title   = {Bounds for the number of nodes in {Chebyshev} type quadrature formulas},
  journal = {Journal of Approximation Theory},
  volume  = {67},
  number  = {2},
  pages   = {199--214},
  year    = {1991},
  doi     = {10.1016/0021-9045(91)90018-6},
  url     = {https://doi.org/10.1016/0021-9045(91)90018-6}
}

@misc{Radchenko2026Stark,
  author  = {},
  title   = {Stark Units for Real Quadratic Fields and Reciprocity Laws},
  year    = {2026},
  month   = oct,
  note    = {Unpublished manuscript, 4 October 2026},
  url     = {https://www.danrad.net/papers/stark-reciprocity.pdf},
  urldate = {2026-10-07}
}

@article{Nasmith2022,
  author  = {Benjamin Nasmith},
  title   = {Octonions and the two strictly projective tight $5$-designs},
  journal = {Algebraic Combinatorics},
  volume  = {5},
  number  = {3},
  pages   = {401--411},
  year    = {2022},
  doi     = {10.5802/alco.215},
  url     = {https://doi.org/10.5802/alco.215}
}

@article{HarpePacheVenkov2007,
  author  = {Pierre de la Harpe and Claude Pache and Boris B. Venkov},
  title   = {Construction of spherical cubature formulas using lattices},
  journal = {St. Petersburg Mathematical Journal},
  volume  = {18},
  number  = {1},
  pages   = {119--139},
  year    = {2007},
  doi     = {10.1090/S1061-0022-07-00946-6},
  url     = {https://doi.org/10.1090/S1061-0022-07-00946-6}
}

@article{BannaiSloane1981,
  author  = {Eiichi Bannai and N. J. A. Sloane},
  title   = {Uniqueness of Certain Spherical Codes},
  journal = {Canadian Journal of Mathematics},
  volume  = {33},
  number  = {2},
  pages   = {437--449},
  year    = {1981},
  doi     = {10.4153/CJM-1981-038-7},
  url     = {https://doi.org/10.4153/CJM-1981-038-7}
}

@article{CohnKumar2007,
  author  = {Cohn, Henry and Kumar, Abhinav},
  title   = {Universally Optimal Distribution of Points on Spheres},
  journal = {Journal of the American Mathematical Society},
  volume  = {20},
  number  = {1},
  pages   = {99--148},
  year    = {2007},
  doi     = {10.1090/S0894-0347-06-00546-7},
  url     = {https://arxiv.org/abs/math/0607446}
}

@incollection{BachocVenkov2001,
  author    = {Bachoc, Christine and Venkov, Boris},
  title     = {Modular Forms, Lattices and Spherical Designs},
  editor    = {Martinet, Jacques},
  booktitle = {R{\'e}seaux euclidiens, designs sph{\'e}riques
               et formes modulaires},
  series    = {Monographies de L'Enseignement Math{\'e}matique},
  volume    = {37},
  pages     = {87--111},
  publisher = {L'Enseignement Math{\'e}matique},
  address   = {Geneva},
  year      = {2001},
  url       = {https://www.unige.ch/math/EnsMath/monographies/monographie-no-37}
}
\end{document}